\documentclass[11 pt]{article}
\usepackage{latexsym}
\usepackage{graphicx}
\usepackage{amsmath}
\usepackage{amsthm}
\usepackage{amssymb,amsfonts,enumerate}
\usepackage[pdfpagelabels=true,plainpages=false,colorlinks=true,linkcolor=blue,citecolor=blue,urlcolor=blue]{hyperref}
\newtheorem{theorem}{Theorem}[section]

\newtheorem{definition}{Definition}[section]
\newtheorem{remark}{Remark}[section]

\begin{document}
\begin{center}
{\Large{\bf New class of sequences of fuzzy numbers defined by modulus function and generalized weighted mean}}\\\vspace{.5cm}
Sarita Ojha and P. D. Srivastava\\
Department of Mathematics, Indian Institute of Technology,\\ Kharagpur 721302, India
\end{center}

\section*{Abstract}
In this paper, We have introduced a new class of sequences of fuzzy numbers defined by using modulus function and generalized weighted mean over the class defined in \cite{OS}. We have proved that this class form a quasilinear complete metric space under a suitable metric. Various inclusion relations and some properties such as solidness, symmetry are investigated.\\
{\bf Keywords :} Sequences of fuzzy numbers; Modulus function; Generalized weighted mean.\\
{\bf AMS subject classification :} 46S40; 03E72.

\section{Introduction}
Zadeh \cite{Z} has introduced the concept of fuzzy sets and fuzzy set operations. Subsequently several authors have discussed various
aspects of its theory and applications such as fuzzy topological spaces, similarity relations and fuzzy orderings, fuzzy measures of
fuzzy events, fuzzy mathematical programming etc. \\
In recent years, there has been an increasing interest in various mathematical aspects of operations defined on fuzzy sets. Based on these, sequences of fuzzy numbers have been introduced by several authors and they have obtained many important properties. Matloka \cite{M} has studied the bounded and convergent sequences of fuzzy numbers where it is shown that every convergent sequence is bounded. The class of all $p$-summable convergent sequences of fuzzy numbers is introduced by Nanda \cite{N}. Later on, Talo and Basar \cite{TB}, \cite{TB1} determined the dual of classical sets of sequences of fuzzy numbers and related matrix transformations. Recently bounded variation for fuzzy numbers is studied by Tripathy et al in \cite{TD}\cite{TD1}.\\
As the set of all real numbers can be embedded in the set of all fuzzy numbers, many results in reals can be considered as a special case of those fuzzy numbers. However, since the set of all fuzzy numbers is partially ordered and does not carry a group structure, most of the facts known for the sequences of real numbers may not be valid in fuzzy setting. Therefore, this theory is not a trivial extension of what has been known in real case.\\
Recently, Polat et al \cite{PKS} have defined sequence spaces of real numbers using generalized weighted mean. This has been extended by Ojha and Srivastava \cite{OS} by introducing new sets of sequences $\lambda^F (u,v)$,  $\lambda=l_p\ (0< p\leq\infty),c,c_0$ of fuzzy numbers. Various properties such as quasilinearity, completeness of these sequences are investigated by them. In the present paper, we have defined a more general class of sequences of fuzzy numbers using modulus function over the class $\lambda^F (u,v)$.

\section{Definition and Preliminaries}
\begin{definition}
A fuzzy real number $X:\mathbb{R} \to [0,1]$ is a fuzzy set satisfying the following conditions:
\begin{enumerate}[(i)]
\item $X$ is normal i.e. there exists an $t\in R$ such that $X(t)=1$.
\item $X$ is fuzzy convex i.e. $X[\lambda t+(1-\lambda)s]\geq \min \{X(t),X(s)\}$ for all $t,s\in R$ and for all $\lambda\in [0,1]$.
\item $X$ is upper semi-continuous.
\item $X^0=\overline{\{t\in R:X(t)>0\}}$ is compact.
\end{enumerate}
\end{definition}
Clearly, $\mathbb{R}$ is embedded in $L(R)$, the set of all fuzzy numbers, in this way: for each $r\in \mathbb{R}$, $\overline{r}\in L(R)$ is defined as,
\begin{center}
$\overline{r}(t) = \left\{
\begin{array}{c l}
  1, & t=r \\
  0, & t\neq r
\end{array}
\right.$
\end{center}
For, $0<\alpha\leq1$, $\alpha$-cut of a fuzzy number $X$ is defined by $X^{(\alpha)}=\{t\in \mathbb{R}:X(t)\geq\alpha\}$. The set $X^{(\alpha)}$ is a closed, bounded and non-empty interval for each $\alpha\in [0,1]$.\\
Also, we denote the triangular fuzzy numbers as $X=[a,b,c]$ i.e. its membership function is defined as,
\begin{center}
$X(t) = \left\{
\begin{array}{c l}
  \frac{x-a}{b-a}, & a\leq t\leq b \\
  \frac{c-x}{c-b}, & b\leq t\leq c
\end{array}
\right.$
\end{center}
\begin{center}
\begin{figure}
\centering
\begin{minipage}{.5\textwidth}
        \includegraphics[height= 4cm, width=5cm]{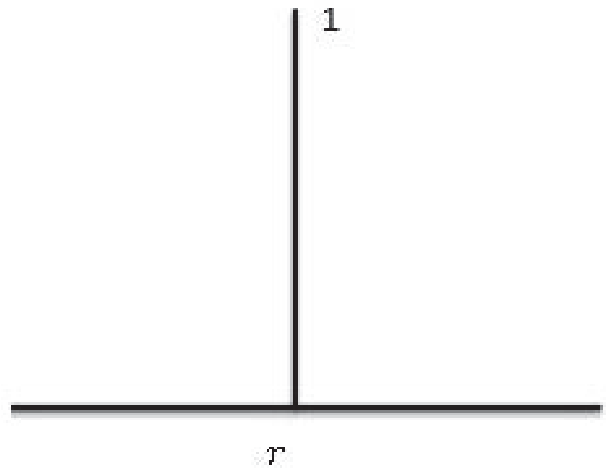}
        \caption{Fuzzy number $\bar{r}$}
  \label{fig:sub1}
\end{minipage}%
\begin{minipage}{.5\textwidth}
        \includegraphics[height= 4cm, width=5cm]{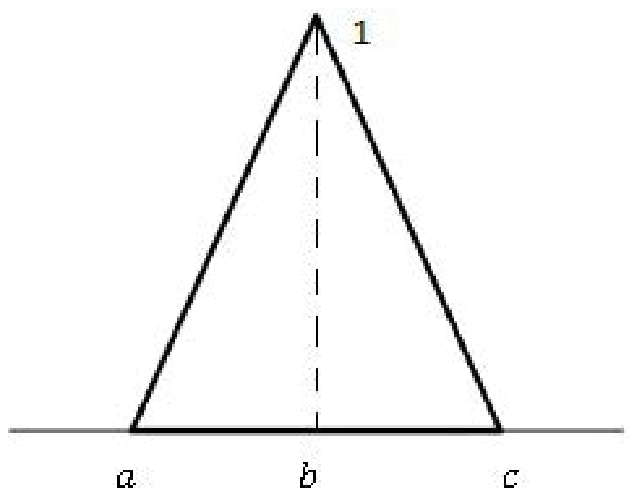}
        \caption{Fuzzy number $[a,b,c]$}
   \label{fig:sub2}
\end{minipage}
\end{figure}
\end{center}
Now for any two fuzzy numbers $X,Y$, Matloka \cite{M} has introduced a metric,
\begin{eqnarray*}
d(X,Y)=\sup\limits_{0\leq \alpha \leq1} \max\{|\underline{X}^{(\alpha)}-\underline{Y}^{(\alpha)}|,|\overline{X}^{(\alpha)}-\overline{Y}^{(\alpha)}|\}
\end{eqnarray*}
where $\underline{X}^{(\alpha)}$ and $\overline{X}^{(\alpha)}$ are the lower and upper bound of the $\alpha$-cut and showed that $(L(R),d)$ is a complete metric space.\\
A partial ordering relation on $L(R)$ is defined as follows:
\begin{eqnarray}
&& X \preceq Y \Leftrightarrow X^{(\alpha)}\preceq Y^{(\alpha)} \nonumber \\
&\Leftrightarrow& \overline{X}^{\alpha}\leq \overline{Y}^{\alpha} \ \mbox{and} \ \underline{X}^{\alpha}\leq \underline{Y}^{\alpha}
\end{eqnarray}
$\mbox{for all} \ \alpha\in [0,1]$. The absolute value of $X\in L(R)$ is defined by,
\begin{center}
$|X|(t) = \left\{
\begin{array}{c l}
  \max\{X(t),X(-t)\}, & t\geq0 \\
  0, & t<0
\end{array}
\right.$
\end{center}

\begin{theorem}(Representation Theorem) \cite{TB}\\ Let $X ^{(\alpha)} = [\underline{X} ^{(\alpha)}, \overline{X} ^{(\alpha)}]$ for $X\in L(R)$ and for
each $\alpha\in[0, 1]$. Then the following statements hold:
\begin{enumerate}[(i)]
\item $\underline{X} ^{(\alpha)}$ is a bounded and non-decreasing left continuous function on (0, 1].
\item $\overline{X} ^{(\alpha)}$ is a bounded and non-increasing left continuous function on (0, 1].
\item The functions $\underline{X} ^{(\alpha)}$ and $\overline{X} ^{(\alpha)}$ are right continuous at the point $\alpha=0$.
\item $\underline{X} ^{(1)} \leq \overline{X} ^{(1)}$.
\end{enumerate}
Conversely, if the pair of functions $P$ and $Q$ satisfies the conditions (i)-(iv), then
there exists a unique $X\in L(R)$ such that $X ^{(\alpha)} = [P(\alpha), Q(\alpha)]$ for each $\alpha\in [0, 1]$.
The fuzzy number $X$ corresponding to the pair of functions $P$ and $Q$ is defined by
$X : R\to[0, 1],\ X(t) = \sup\{\alpha : P(\alpha)\leq t \leq Q(\alpha)\}$.
\end{theorem}

\noindent From the above theorem, it is clear that for any $X\in L(R)$, $d(X,\bar{0})=\max\{|\underline{X^0}|,|\overline{X^0}|\}$. Let $w^F$ be the set of sequences all fuzzy numbers.

\begin{definition} Let $\tau^F\subset w^F$. Then the set $\tau^F$ is said to be
\begin{enumerate}
\item Solid if $(Y_k)\in w^F$, whenever $(X_k)\in w^F$ and $|Y_k|\preceq |X_k|$, for all $k\in\mathbb{N}$.
\item Symmetric if $(X_{\pi (k)})\in w^F$, whenever $(X_k)\in w^F$, where $\pi$ is a permutation on $\mathbb{N}$.
\item Convergence free if $(Y_k)\in w^F$, whenever $(X_k)\in w^F$ and $X_k=\bar{0}$ implies $Y_k=\bar{0}$.
\end{enumerate}
\end{definition}
Throughout this paper, $U$ be the set of all real sequences $u=(u_n)$, $n=0,1,2,...$ such that for all $n$, $u_n\neq0$. Let $u,v\in U$ and consider the matrix $G(u,v)=(g_{nk})$ as
\begin{center}
$g_{nk} = \left\{
\begin{array}{c l}
  u_nv_k, & 0\leq k\leq n \\
  0, & k>n
\end{array}
\right.$
\end{center}
\noindent for all $k,n\in\mathbb{N}$ where $u_n$ depends only on $n$ and $v_k$ depends only on $k$. This matrix $G(u,v)$ is called generalized weighted mean or factorable matrix. Many authors such as Altay and Basar \cite{AB}, Polat et al. \cite{PKS} etc constructed real sequence spaces by using this generalized weighted mean. Recently in 2014, Ojha and Srivastava \cite{OS} used this matrix to introduce a new class of sequences of fuzzy numbers.\\
In 2014, Ojha and Srivastava \cite{OS} introduced a new class defined as,
\begin{eqnarray*}
\lambda^F(u,v)=\Big\{X=(X_k)\in w^F: \Big(\sum\limits_{i=0} ^k u_kv_id(X_i,\bar{0})\Big)\in\lambda\Big\}
\end{eqnarray*}
where $\lambda=l_\infty,c,c_0,l_p\ (p>1)$ and $u,v\in U$. In the present paper, we have defined the set of sequences of fuzzy numbers by using modulus function and
generalized weighted mean over $\lambda^F(u,v)$ and investigated various properties of it.

\begin{remark}
As special cases, we get Cesaro mean, weighted mean of fuzzy numbers etc. eg.
\begin{enumerate}[(i)]
\item For $u_k=\frac{1}{k}$ and $v_k=1$ for all $k\in\mathbb{N}$, we get $\sum\limits_{i=0} ^k u_kv_id(X_i,\bar{0})=\frac{1}{k} \sum\limits_{i=0} ^k d(X_i,\bar{0})$
which is Cesaro sequence of fuzzy numbers.
\item  If we take $v_i$'s as positive numbers and $u_k=\frac{1}{v_1+v_2+\cdots+v_k}$, then
\begin{equation*}
\frac{v_1d(X_1,\bar{0})+v_2d(X_2,\bar{0})+\cdots+v_kd(X_k,\bar{0})}{v_1+v_2+\cdots+v_k}
\end{equation*} is the weighted mean over $(L(R),d)$.
\end{enumerate}
\end{remark}

\noindent Ruckle \cite{R} introduced the concept of modulus function and also defined a new sequence spaces by using modulus function of real or complex numbers.
\begin{definition}
A modulus function  $f$ is a function $f: [0,\infty) \to [0,\infty)$ such that,
\begin{enumerate}[(i)]
\item $f(x)=0$ iff $x=0$.
\item $f(x+y)\leq f(x)+f(y)$.
\item $f$ is increasing.
\item $f$ is continuous from the right at 0.
\end{enumerate}
\end{definition}
It follows that $f$ must be continuous everywhere on $[0,\infty)$.  Using the modulus function, Maddox \cite{Mad}\cite{Mad1} introduced the generalizations
of the classical spaces of strongly summable sequences by defining certain sequence spaces using a modulus function $f$. \\

\noindent Let $r=(r_k)$ be a bounded sequence of strictly positive real numbers. If $H=\sup\limits_{k} r_k$, then for any real or complex numbers $a_k, b_k$, we have
\begin{eqnarray*}
|a_k+b_k|^{r_k}\leq D (|a_k|^{r_k}+|b_k|^{r_k})\\
\mbox{and}\ |c|^{r_k}\leq\max(1,|c|^H)
\end{eqnarray*}
where $D=\max (1,2^{H-1})$ and $c$ is any complex number.

\section{Main results}
Let $u,v\in U$ and $r=(r_k)$ is a bounded sequence of strictly positive real numbers. We define a new set of sequences of fuzzy numbers using the modulus function $f$ as follows:
\begin{equation*}
\lambda^F [u,v,f,r] = \Big\{X=(X_k)\in w^F:\Big[f\Big(\Big|u_k\sum\limits_{i=0}^k v_id(X_i,\bar{0})\Big|\Big)\Big]^{r_k}\in \lambda\Big\}
\end{equation*}
for $\lambda=l_p\ (0<p\leq\infty),c_0$ and in case of $\lambda=c$, we define the class as follows:
\begin{eqnarray*}
c^F [u,v,f,r] = \Big\{X=(X_k)\in w^F:\Big[f\Big(\Big|u_k\sum\limits_{i=0}^k v_id(X_i,X_0)\Big|\Big)\Big]^{r_k}\to 0 \ \mbox{for}\\ \mbox{some}\ X_0\in L(R)\Big\}
\end{eqnarray*}

\noindent {\bf Note :} In case of $f(x)=x$ and $r_k=1$, the above class coincides with the class defined by Ojha and Srivastava \cite{OS}.

\begin{theorem}
The set $\lambda^F [u,v,f,r]$ is a quasilinear space for $\lambda=l_p\ (0<p\leq\infty),c_0,c$.
\end{theorem}
\begin{proof}
Since $\lambda^F [u,v,f,r]\subset L(R)$, so all the properties as listed in \cite{A} trivially follows from the property of $L(R)$. So it is enough to show that $\lambda^F [u,v,f,r]$ is closed under addition and scalar multiplication which follows from the properties of the modulus function and the metric $d$ on $L(R)$. Therefore, we omit the details of the proof.
\end{proof}

\begin{theorem}
The set $\lambda^F [u,v,f,r]$ for  is a complete metric space under the following metric,
\begin{eqnarray*}
D(X,Y) &=& \sup_k\Big[f\Big(\Big|u_k\sum\limits_{i=0}^k v_id(X_i,Y_i)\Big|\Big)\Big]^{r_k/M}\ \mbox{for $\lambda=l_\infty,c,c_0$}\\
D_p(X,Y) &=& \Big\{\sum\limits_{k=0}^{\infty}\Big[f\Big(\Big|u_k\sum\limits_{i=0}^k v_id(X_i,Y_i)\Big|\Big)\Big]^{r_k}\Big\}^{1/M}\ \mbox{for $\lambda=l_p$ $(0<p<\infty)$}
\end{eqnarray*}
where $M=\max\{1,\sup\limits_k r_k\}$ and $X, Y\in \lambda^F [u,v,f,r]$.
\end{theorem}
\begin{proof}
We shall give the proof for $\lambda=l_{\infty}$. Rest will follow the same lines.
It is easy to verify that $D$ is a metric on $l_{\infty} ^F [u,v,f,r]$. To prove the completeness property, let $(X^n)=(X_i ^n)$ be a Cauchy sequences in $l_{\infty} ^F [u,v,f,r]$. Then,
\begin{eqnarray*}
\sup_k\Big[f\Big(\Big|u_k\sum\limits_{i=0}^k v_id(X_i ^n,X_i ^m)\Big|\Big)\Big]^{r_k/M}\to 0 \ \mbox{as}\ n,m\to\infty \\
\Rightarrow \ f\Big(\Big|u_k\sum\limits_{i=0}^k v_id(X_i ^n, X_i ^m)\Big|\Big)\to 0 \ \mbox{as}\ n,m\to\infty\ \mbox{and for all}\ k
\end{eqnarray*}
This follows because each term is non-negative and $(r_k)$ is a bounded sequence of non-negative real numbers. Since $f$ is modulus function, so for all $k$, we have
\begin{eqnarray*}
u_k\sum\limits_{i=0}^k v_id(X_i ^n, X_i ^m)\to 0 \ \mbox{as}\ n,m\to\infty\\
\mbox{Hence,} \ d(X_i ^n, X_i ^m)\to 0 \ \mbox{as}\ n,m\to\infty
\end{eqnarray*}
which shows that $(X_i ^n)$ is a Cauchy sequence in $L(R)$, hence it is convergent. Suppose $X_i ^n \to X_i$ for all $i$ as $n\to\infty$. So, $X^n\to X=(X_i)$. Now, it is remained to show that $X\in l_{\infty} ^F [u,v,f,r]$.
\begin{eqnarray*}
\mbox{Consider,}\ \Big|u_k\sum\limits_{i=0}^k v_id(X_i,\bar{0})\Big| &\leq& \Big|u_k\sum\limits_{i=0}^k v_id(X_i,X_i ^n)\Big|+\Big|u_k\sum\limits_{i=0}^k v_id(X_i ^n,\bar{0})\Big|\\
\Rightarrow\ f\Big(\Big|u_k\sum\limits_{i=0}^k v_id(X_i,\bar{0})\Big|\Big) &\leq& f\Big(\Big|u_k\sum\limits_{i=0}^k v_id(X_i,X_i ^n)\Big|\Big)
+f\Big(\Big|u_k\sum\limits_{i=0}^k v_id(X_i ^n,\bar{0})\Big|\Big)\\
&& \mbox{(using the properties of modulus function $f$)}\\
\Rightarrow\ \Big[f\Big(\Big|u_k\sum\limits_{i=0}^k v_id(X_i,\bar{0})\Big|\Big)\Big]^{r_k} &\leq& B\Big[f\Big(\Big|u_k\sum\limits_{i=0}^k v_id(X_i,X_i ^n)\Big|\Big)\Big]^{r_k}
\\ && +B\Big[f\Big(\Big|u_k\sum\limits_{i=0}^k v_id(X_i ^n,\bar{0})\Big|\Big)\Big]^{r_k}
\end{eqnarray*}
where $B$ is a constant as defined earlier. Since $(X_i ^n)\in l_{\infty} ^F [u,v,f,r]$ and $X^n _i \to X_i$ under the metric $D$, so $X\in l_{\infty} ^F [u,v,f,r]$ i.e. $l_{\infty} ^F [u,v,f,r]$ is a complete metric space.
\end{proof}

\begin{theorem}
$\lambda^F [u,v,f,r]$ is a $K$-space for $\lambda=c,c_0,l_p$.
\end{theorem}
\begin{proof}
It is enough to show it under the metric $D$. Others will follow similar lines.\\
Here we have to show that the projection operator
\begin{equation*}
P_k:\lambda^F [u,v,f,r]\to L(R)\ \mbox{given by} \ P_k(X)=X_k,\  X\in \lambda^F [u,v,f,r]
\end{equation*}
is continuous for each $k\in\mathbb{N}$ i.e. it is enough to show that if $(X^n)$ be a sequence in $\lambda^F [u,v,f,r]$
such that $X^n\to \bar{0}$ in $\lambda^F [u,v,f,r]$, then $X^n _i\to \bar{0}$ in $L(R)$ $\forall\ i$.
Let, $X^n\to \bar{0}$ in $\lambda^F [u,v,f,r]$, then we have
\begin{eqnarray*}
D(X^n,\bar{0})=\sup_k \Big[f\Big(\Big|u_k\sum\limits_{i=0}^k v_id(X_i ^n,\bar{0})\Big|\Big)\Big]^{r_k/M}\to 0
\end{eqnarray*}
$\mbox{as}\ n\to\infty$. Which gives that, $\sum\limits_{i=0}^k u_kv_id(X^n _i,\bar{0})\to 0$ for each $k$ as $n\to\infty$.
Proceeding with similar arguments as before, we get
\begin{equation*}
d(X^n _i,\bar{0})\to 0 \ \mbox{for all $i$ as}\ n\to\infty.
\end{equation*}
So, $X^n _i\to \bar{0}$ in $L(R)$ i.e. $\lambda^F [u,v,f,r]$ is a $K$-space.
\end{proof}

\begin{theorem}
The sequence space $\lambda^F [u,v,f,r]$ is solid for $\lambda=l_p,c_0$.
\end{theorem}
\begin{proof}
Assume $(X_k)\in \lambda^F [u,v,f,r]$ for the above choice of $\lambda$ and $(Y_k)$ is a sequence of fuzzy numbers such that $d(Y_k,\bar{0})\leq d(X_k,\bar{0})$. Then,
\begin{eqnarray*}
\Big|u_k\sum\limits_{i=0}^k v_id(Y_i,\bar{0})\Big| &\leq& \Big|u_k\sum\limits_{i=0}^k v_id(X_i,\bar{0})\Big|\\
\mbox{Since $f$ is increasing, so}\ f\Big(\Big|u_k\sum\limits_{i=0}^k v_id(Y_i,\bar{0})\Big|\Big) &\leq& f\Big(\Big|u_k\sum\limits_{i=0}^k v_id(X_i,\bar{0})\Big|\Big)
\end{eqnarray*}
Thus we have, $(Y_k)\in \lambda ^F [u,v,f,r]$ whether $\lambda=l_{\infty}$ or $c_0$. So, $\lambda ^F [u,v,f,r]$ is solid.
\end{proof}

\begin{remark}
$c^F [u,v,f,r]$ is not solid. To show this, consider the following example.\\
Assume $f(x)=x$ and $r_k=1,u_k=1,v_k=(-1)^k$ for all $k\geq1$. Let us take the sequences $(X_i)$ and $(Y_i)$ as, for all $i\geq1$,
\begin{eqnarray*}
X_i = \overline{\Big(\frac{1}{i}\Big)}\ \mbox{and}\
Y_i = \left\{
\begin{array}{c l}
 \overline{\Big(\frac{1}{i}\Big)}, & i \ \mbox{even} \\
  \bar{0}, & i \ \mbox{odd}\\
  \end{array}
\right.
\end{eqnarray*}
Then it is clear that $d(Y_i,\bar{0})\leq d(X_i,\bar{0})$ and $(X_k)\in c^F [u,v,f,r]$ but $(Y_k)\notin c^F [u,v,f,r]$.
\end{remark}

\begin{theorem}
$\lambda ^F [u,v,f,r], \lambda=c,c_0,l_p$ is not symmetric.
\end{theorem}
\begin{proof}
We shall prove this result for $c_0$.\\
Consider $f(x)=x$ and $r_k=1$ for all $k$. Let $u_k=\frac{1}{k}$ and $v_k=1$ for all $k\geq1$.
Let us define,
\begin{center}
$X_i = \left\{
\begin{array}{c l}
  \bar{1}, & i=n^2 \ \mbox{for some integer}\ n \\
  \bar{0}, & \mbox{otherwise}\\
  \end{array}
\right.\ \mbox{and}\
Y_i = \left\{
\begin{array}{c l}
  \bar{1}, & i \ \mbox{odd} \\
  \bar{0}, & i \ \mbox{even}\\
  \end{array}
\right.$
\end{center}
Then,
\begin{eqnarray*}
\Big|\sum\limits_{i=1}^k u_kv_id(X_i,\bar{0})\Big|=\frac{1}{k} [\sqrt{k}] \leq \frac{1}{\sqrt{k}} \to 0\\
\Big|\sum\limits_{i=1}^k u_kv_id(Y_i,\bar{0})\Big|=\frac{1}{k} \frac{k}{2} =\frac{1}{2}\nrightarrow 0
\end{eqnarray*}
where $[\cdot]$ is the box function. Therefore $(X_k)\in \lambda ^F [u,v,f,r]$ whereas $(Y_k)\notin \lambda ^F [u,v,f,r]$.\\
In a similar way, we can show for $\lambda=c,l_p$ also. So, $\lambda ^F [u,v,f,r]$ is not symmetric.
\end{proof}

\begin{theorem}
$\lambda ^F [u,v,f,r]$ does not form a sequence algebra for $\lambda=c,c_0,l_p$.
\end{theorem}
\begin{proof}
To verify this for $\lambda=l_p$, Let us assume $u_k=\frac{1}{k^4}, v_k=\frac{1}{k}$ and $r_k=1$ for all $k\geq1$ and. Also let $f$ be a identity function on $[0,\infty)$ to $[0,\infty)$. Let $(X_k)$ be a sequence defined as, $X_k=[-k^2,0,k^2]$. Then
\begin{eqnarray*}
f(u_k\sum\limits_{i=1} ^k v_id(X_i,\bar{0})) &=& \frac{1}{k^4}\sum\limits_{i=1} ^k \frac{1}{i} \cdot i^2 \\
&=& \frac{1}{k^4} \frac{k(k+1)}{2}\\
&=& \frac{1}{2} \Big[\frac{1}{k^2}+\frac{1}{k^3}\Big]
\end{eqnarray*}
Since $\sum\limits_{k=1} ^{\infty}f(u_k\sum\limits_{i=1} ^k v_id(X_i,\bar{0}))<\infty$  $\therefore\ (X_k)\in l_p ^F [u,v,f,r]$.\\
Let $(Y_k)$ be an another sequence defined by $Y_k=[-k,0,k]$. Then
\begin{eqnarray*}
f(u_k\sum\limits_{i=1} ^k v_id(Y_i,\bar{0})) &=& \frac{1}{k^4}\sum\limits_{i=1} ^k \frac{1}{i} \cdot i\\
&=& \frac{1}{k^4} k\\
&=& \frac{1}{k^3}
\end{eqnarray*}
By similar arguments, $(Y_k)\in l_p ^F [u,v,f,r]$.\\
Now, $X_kY_k=[-k^3,0,k^3]$ and therefore
\begin{eqnarray*}
f(u_k\sum\limits_{i=1} ^k v_id(X_iY_i,\bar{0})) &=& \frac{1}{k^4}\sum\limits_{i=1} ^k \frac{1}{i} \cdot i^3 \\
&=& \frac{1}{k^4} \frac{k(k+1)(k+2)}{6}\\
&=& \frac{1}{6} \Big[\frac{1}{k}+\frac{3}{k^2}+\frac{2}{k^3}\Big]
\end{eqnarray*}
Since, $\sum\limits_{k=1} ^{\infty}f(u_k\sum\limits_{i=1} ^k v_id(X_iY_i,\bar{0}))=\infty$. So, $(X_kY_k)\notin l_p ^F [u,v,f,r]$. Therefore $\lambda ^F [u,v,f,r]$ does not form a sequence algebra.\\
Rests can be proved in similar way.
\end{proof}

\begin{theorem}
$\lambda ^F [u,v,f,r]$ is not convergence free for $\lambda=c,c_0,l_p$.
\end{theorem}
\begin{proof}
We shall show it for $\lambda=c_0$ by constructing a counter example. For this let us take $f(x)=x$ and $u_k=\frac{1}{k^2}$,  $v_k=k^2$ for all $k\geq1$. Also take the sequences $(X_k)$ and $(Y_k)$ as follows:
\begin{center}
$X_k = \left\{
\begin{array}{c l}
  [-\frac{1}{k^2}, 0, \frac{1}{k^2}], & k \ \mbox{is odd} \\
  \bar{0}, & k\ \mbox{is even}\\
  \end{array}
\right.\ \mbox{and}\
Y_k = \left\{
\begin{array}{c l}
  [-\frac{1}{k}, 0, \frac{1}{k}], & k \ \mbox{is odd} \\
  \bar{0}, & k \ \mbox{is even}\\
  \end{array}
\right.$
\end{center}
\begin{eqnarray*}
\mbox{Then, } u_k\sum\limits_{i=1} ^k v_id(X_i,\bar{0}) &=& \frac{1}{k^2}\cdot\Big[\frac{k}{2}\Big]\leq\frac{1}{2k}\to 0 \ \mbox{as $k\to\infty$} \\
\mbox{and } u_k\sum\limits_{i=1} ^k v_id(Y_i,\bar{0}) &=& \left\{
\begin{array}{c l}
  \frac{1}{k^2}[1+3+5+\cdots+(k-1)], & k \ \mbox{is even} \\
  \frac{1}{k^2}[1+3+5+\cdots+k], & k\ \mbox{is odd}\\
  \end{array}
\right.\\
&=& \left\{
\begin{array}{c l}
  \frac{1}{k^2}(k-1)^2, & k \ \mbox{is even} \\
  \frac{1}{k^2}k^2, & k\ \mbox{is odd}\\
  \end{array}
\right.
\end{eqnarray*}
which does not tend to 0 as $k\to\infty$. So, $(X_k)\in c_0 ^F[u,v,f,r]$ but $(Y_k)\notin c_0 ^F[u,v,f,r]$ i.e. $c_0 ^F[u,v,f,r]$ is not convergence free. By similar way, we can show it for $\lambda=c,l_p$ also.
\end{proof}

\subsection{Inclusion relations}

\begin{theorem}
$l_1 ^F[u,v,f,r]\subset c_0 ^F[u,v,f,r]\subset c^F[u,v,f,r]\subset  l_{\infty} ^F[u,v,f,r]$. The inclusions are in strict sense.
\end{theorem}
\begin{proof}
Proof is easy, so we omit it. To prove the strict sense of the above inclusion relations, let us take the counter examples,
\begin{enumerate}[(i)]
\item Let $u_k=\frac{1}{k^2}$ and $v_k=1$ for all $k$. Also take $X_i=\bar{1}$ for all $i\geq1$. Then $\sum\limits_{i=1} ^k v_id(X_i,\bar{0})=k$. So, it is easy to see $(X_k)\in c_0 ^F[u,v,f,r]$ but $(X_k)\notin l_1 ^F[u,v,f,r]$.
\item $u_k=\frac{1}{k^2},v_k=k^2$ for all $k$ and let
$X_k = \left\{
\begin{array}{c l}
  [-\frac{1}{k}, 0, \frac{1}{k}], & k \ \mbox{is odd} \\
  \bar{0}, & k \ \mbox{is even}\\
  \end{array}
\right.$\\
Then by the previous example, we can show that $(X_k)\in c^F(u,v,f,r)$ as it converges to 1 as $k\to\infty$ but $(X_k)\notin c_0 ^F(u,v,f,r)$
\item Let $u_k=1,v_k=(-1)^k$ and $(X_k)=\bar{1}$ for all $k$. Then $u_k\sum\limits_{i=1} ^k v_id(X_i,\bar{0})=0$ or 2 according to $k$ is even or odd. So, $(X_k)\in l_{\infty} ^F(u,v,f,r)$ but $(X_k)\notin c ^F(u,v,f,r)$.
\end{enumerate}
This completes the proof.
\end{proof}

\begin{theorem}
Let $0<r_k\leq q_k<\infty$ for all $k$. Then for any modulus function $f$
\begin{enumerate}[(i)]
\item $\lambda^F [u,v,f,r]\subseteq \lambda^F [u,v,f,q]$ for $\lambda=c,c_0,l_1$.
\item If $(q_k/r_k)$ is bounded, then $l_{\infty} ^F [u,v,f,r]\subseteq l_{\infty} ^F [u,v,f,q]$.
\end{enumerate}
\end{theorem}
\begin{proof}
(i) is easy to prove. So we omit it. Now to prove (ii) let us assume that $X\in l_{\infty} ^F [u,v,f,r]$, then we have
\begin{eqnarray*}
\sup_k\Big[f\Big(\Big|u_k\sum\limits_{i=0}^k v_id(X_i,\bar{0})\Big|\Big)\Big]^{r_k}<M
\end{eqnarray*}
for some $M>0$. $t_k=\Big[f\Big(\Big|u_k\sum\limits_{i=0}^k v_id(X_i,\bar{0})\Big|\Big)\Big]^{r_k}$. Now define
\begin{center}
$Y_k = \left\{
\begin{array}{c l}
  X_k, & \mbox{when }t_k<1 \\
  0, & \mbox{otherwise}
\end{array}
\right.$
and $Z_k = \left\{
\begin{array}{c l}
  X_k, & \mbox{when }t_k\geq1 \\
  0, & \mbox{otherwise}
\end{array}
\right.$
\end{center}
\begin{eqnarray*}
\mbox{So,}\ X_k=Y_k+Z_k\ \mbox{i.e. } d(X_k,\bar0) &\leq& d(Y_k,\bar{0})+d(Z_k,\bar{0})\\
\mbox{i.e.}\ \Big[f\Big(\Big|u_k\sum\limits_{i=0}^k v_id(X_i,\bar{0})\Big|\Big)\Big]^{q_k} &=& B\Big[f\Big(\Big|u_k\sum\limits_{i=0}^k v_id(Y_i,\bar{0})\Big|\Big)\Big]^{q_k}\\ && +B\Big[f\Big(\Big|u_k\sum\limits_{i=0}^k v_id(Z_i,\bar{0})\Big|\Big)\Big]^{q_k}\\ && \mbox{where $B$ is a constant.}\\
&\leq& B\Big[f\Big(\Big|u_k\sum\limits_{i=0}^k v_id(Y_i,\bar{0})\Big|\Big)\Big]^{r_k}\\ && +B\Big\{\Big[f\Big(\Big|u_k\sum\limits_{i=0}^k v_id(Z_i,\bar{0})\Big|\Big)\Big]^{r_k}\Big\}^{q_k/ r_k}\\
&\leq& B[M+M^{q_k/r_k}]
\end{eqnarray*}
Since $(q_k/r_k)$ is bounded, so taking supremum over $k$ in the above inequality, the R.H.S will always be finite. Therefore $(X_k)\in l_{\infty} ^F [u,v,f,q]$. Thus $l_{\infty} ^F [u,v,f,r]\subseteq l_{\infty} ^F [u,v,f,q]$.
\end{proof}

\begin{theorem}
Let $f,g$ be two modulus functions. Then the following relations hold:
\begin{enumerate}[(i)]
\item If $f(t)\leq g(t)$ for all $t\in\mathbb{R}^+$, then $\lambda^F [u,v,g,r]\subseteq \lambda^F [u,v,f,r]$ for $\lambda=c,c_0,l_1, l_{\infty}$.
\item $\lambda^F [u,v,f,r]\cap \lambda^F [u,v,g,r]\subset\lambda^F [u,v,f+g,r]$ for $\lambda=c,c_0,l_1, l_{\infty}$.
\item $\lambda^F [u,v,f,r]\subset \lambda^F [u,v,g\circ f,r]$ for $\lambda=c,c_0$.
\item If $\lim\limits_{t\to\infty} \frac{f(t)}{t}>0$, then $\lambda^F (u,v,f^n,r)= \lambda^F (u,v,r)$ for some positive integer $n$ for $\lambda=c,c_0$.
\end{enumerate}
\end{theorem}
\begin{proof}
\begin{enumerate}[(i)]
\item is easy to check. so we omit it.
\item For any two functions $f,g$, we have,
\begin{eqnarray*}
(f+g)\Big(\Big|u_k\sum\limits_{i=0}^k v_id(X_i,\bar{0})\Big|\Big)=f\Big(\Big|u_k\sum\limits_{i=0}^k v_id(X_i,\bar{0})\Big|\Big)+g\Big(\Big|u_k\sum\limits_{i=0}^k v_id(X_i,\bar{0})\Big|\Big)
\end{eqnarray*}
which easily gives $\lambda^F [u,v,f,r]\cap \lambda^F [u,v,g,r]\subset\lambda^F [u,v,f+g,r]$.
\item From the continuity of $g$, we have for a pre-assigned positive $\epsilon>0$, $\exists \ \delta>0$ such that $g(\delta)<\epsilon$.
Now assume $X\in c_0 ^F [u,v,f,r]$. Then $f\Big(\Big|u_k\sum\limits_{i=0}^k v_id(X_i,\bar{0})\Big|\Big)<\delta$ for all $k\geq k_0$ for some fixed $k_0\in\mathbb{N}$. Thus we get
    \begin{equation*}
    g\Big(f\Big(\Big|u_k\sum\limits_{i=0}^k v_id(X_i,\bar{0})\Big|\Big)\Big)<g(\delta)<\epsilon.
    \end{equation*}
    So, $X\in c_0 ^F [u,v,g\circ f,r]$. Similarly, we can prove it for $\lambda=c$ also.
\item Putting $f=I$, Identity function on $[0,\infty)$ and $g=f^n$ for some positive integer $n$ in (iii), we got $\lambda^F [u,v,r]\subseteq \lambda^F [u,v,f^n,r]$. To prove the converse, let $\beta=\lim\limits_{t\to\infty} \frac{f(t)}{t}$. Then from \cite{Mad1}, $\beta=\inf\{\frac{f(t)}{t};t>0\}$. So there is a $\beta>0$ such that $f(t)\geq \beta t$ for all $t\geq0$.\\
    So $f^2(t)\geq f(\beta t)\geq \beta\cdot\beta t=\beta^2 t$. Generalizing we get $f^n(t)\geq \beta^n t$ for some positive $n$. Assume $(X_k)\in \lambda^F (u,v,f^n,r)$. Now putting $t=|u_k\sum\limits_{i=0}^k v_id(X_i,\bar{0})|$,
\begin{eqnarray*}
\Big(\Big|u_k\sum\limits_{i=0}^k v_id(X_i,\bar{0})\Big|\Big)^{r_k} &\leq& \Big[\beta^{-n}f^n\Big(\Big|u_k\sum\limits_{i=0}^k v_id(X_i,\bar{0})\Big|\Big)\Big]^{r_k}\\
&\leq& \max\{1,\beta^{-nH}\} \Big[f^n\Big(\Big|u_k\sum\limits_{i=0}^k v_id(X_i,\bar{0})\Big|\Big)\Big]^{r_k}\\
\mbox{i.e.}\ (X_k) &\in & \lambda^F (u,v,r)
\end{eqnarray*}
Combining we get $\lambda^F (u,v,f^n,r)= \lambda^F (u,v,r)$.
\end{enumerate}
\end{proof}

\begin{remark}
\begin{enumerate}[(i)]
\item In theorem 2.9(iii), if $g$ be bounded, then the inclusion relation holds for $\lambda=l_{\infty}$ also.
\item If $f$ be bounded in theorem 2.9(iv), then the inclusion relation is true for $\lambda=l_{\infty}$ also.
\end{enumerate}
\end{remark}

\bibliographystyle{model1-num-names}

\end{document}